\documentclass[11pt, eqno]{article}
\usepackage{bbm}
\usepackage{mathrsfs}
\usepackage{amsfonts}
\usepackage{amssymb}
\usepackage{graphicx}
\usepackage[all]{xy}
\usepackage{amsthm}
\usepackage{amsmath}
\usepackage{amsmath,amssymb,latexsym,color}
\usepackage[mathscr]{eucal}
\usepackage{CJK}
\usepackage{cases}
\usepackage{graphics}

\usepackage{psfrag}
\usepackage{subfigure}
\usepackage{color}

\usepackage{amssymb,latexsym}
\usepackage{amsmath,latexsym}
\usepackage{amscd}

\newcommand{\R}{{\mathbb R}}

\newtheorem{theorem}{Theorem}[section]

\newtheorem{corollary}[theorem]{Corollary}

\newtheorem{remark}[theorem]{Remark}

\newtheorem{lemma}[theorem]{Lemma}

\newtheorem{guess}[theorem]{Conjecture}

\newtheorem{claim}[theorem]{Claim}

\begin{document}
\title{The Viterbo's capacity conjectures for convex toric domains and
 the product of a $1$-unconditional convex body and its polar}
\date{August 20, 2020}
\author{Kun Shi and Guangcun Lu
\thanks{Corresponding author
\endgraf \hspace{2mm} Partially supported
by the NNSF  11271044 of China.
\endgraf\hspace{2mm} 2020 {\it Mathematics Subject Classification.}
52B60, 53D35, 52A40, 52A20.}}
 \maketitle \vspace{-0.3in}


\abstract{
In this note, we  show that the strong Viterbo conjecture holds true on any convex toric domain,
and that the Viterbo's volume-capacity conjecture  holds for the product of
a $1$-unconditional convex body $A\subset\mathbb{R}^{n}$ and its polar. We also
give a direct calculus  proof of the symmetric Mahler conjecture for $l_p$-balls.
  } \vspace{-0.1in}
\medskip\vspace{12mm}

\maketitle \vspace{-0.5in}



\section{Introduction and results}
\setcounter{equation}{0}

Prompted by Gromov's seminal work \cite{Gr85} Ekeland and Hofer  \cite{EH89}
defined a  \textsf{symplectic capacity} on the $2n$-dimensional Euclidean space $\mathbb{R}^{2n}$
with the standard symplectic structure $\omega_0$
to be a map $c$ which associates to each subset $U\subset \mathbb{R}^{2n}$ a number
 a number $c(U)\in[0,\infty]$  satisfying the following axioms:
\begin{description}
\item[(Monotonicity)] $c(U)\le c(V)$ for $U\subset V\subset\mathbb{R}^{2n}$;
\item[(Conformality)] $c(\psi(U))=|\alpha|c(U)$ for $\psi\in{\rm Diff}(\mathbb{R}^{2n})$
such that $\psi^\ast\omega_0=\alpha\omega_0$ with $\alpha\ne 0$;
\item[(Nontriviality)] $0<c(B^{2n}(1))$ and $c(Z^{2n}(1))<\infty$,
where $B^{2n}(r)=\{z\in\mathbb{R}^{2n}\,|\, |z|^2<r^2\}$ and $Z^{2n}(R)=B^2(R)\times\mathbb{R}^{2n-2}$.
\end{description}
Moreover, such a symplectic capacity is called \textsf{normalized} if it also satisfies
\begin{description}
\item[(Normalization)] $c(B^{2n}(1))=c(Z^{2n}(1))=\pi$.
\end{description}

(
Without special statements we make \textsf{conventions}:  1)
symplectic capacities on $\mathbb{R}^{2n}$ are all concerning the  symplectic structure $\omega_0$;
2)  a ``\textsf{domain}" in a Euclidean space always denotes the closure of an open subset;
3) the notation $\langle\cdot,\cdot\rangle$ always denotes the Euclidean inner product.)

 Hofer and Zehnder \cite{HZ90} extended the concept of a symplectic capacity to
 general symplectic manifolds. The first example of a normalized symplectic capacity is
 the Gromov  width $w_G$, which maps a $2n$-dimensional symplectic manifold $(M,\omega)$ to
 \begin{equation}\label{e:GromovWidth}
	w_{\rm G}(M,\omega)=\sup\{\pi r^2\,|\,\exists\textrm{ a symplectic embedding } (B^{2n}(r),\omega_{0})\hookrightarrow (M,\omega)\}.
\end{equation}
In particular, for a subset $U\subset \mathbb{R}^{2n}$ it can be easily proved that
$$
w_{\rm G}(U):=w_{\rm G}(U,\omega_0)=\sup\{\pi r^2\,|\,\exists\;\psi\in{\rm Symp}(\mathbb{R}^{2n})\;
\textrm{with}\;\psi(B^{2n}(r))\subset U\}
$$
with the Extension after Restriction Principle for symplectic embeddings of bounded starshaped open domains
(see Appendix A in \cite{Sch05}). Clearly
$$
c^{\rm Z}(U):=\sup\{\pi r^2\,|\,\exists\;\psi\in{\rm Symp}(\mathbb{R}^{2n})\;
\textrm{with}\;\psi(U)\subset Z^{2n}(r))\}
$$
defines a normalized symplectic capacity on $\mathbb{R}^{2n}$, the so-called \textsf{cylindrical capacity}.
Nowadays, a variety of normalized symplectic capacities can be constructed
in categories of symplectic manifolds for the study of different problems, for example,
the (first) Ekeland-Hofer capacity $c_{\rm EH}$ (\cite{EH89}), the Hofer-Zehnder capacity
$c_{\rm HZ}$ (\cite{HZ90}) and Hofer's displacement energy $e$ (\cite{H90}), the
Floer-Hofer capacity $c_{\rm FH}$ (\cite{FHW94}) and Viterbo's generating function capacity $c_{\rm V}$ (\cite{Vi92})),
the first Gutt-Hutchings capacity $c_1^{\rm CH}$ (\cite{GuH18}) coming from $S^1$-equivariant symplectic homology,
 and the first ECH capacity $c_1^{\rm ECH}$ in dimension $4$ (\cite{Hut011}).
 Except the last $c_1^{\rm ECH}$ the others have defined for all convex domains in $(\mathbb{R}^{2n},\omega_0)$.
 As an immediate consequence of the normalization axiom we see that
$w_{\rm G}$ and $c^{\rm Z}$ are the smallest and largest normalized symplectic capacities on $\mathbb{R}^{2n}$, respectively.
 An important open question in symplectic topology (\cite{Os14, MS2017}), termed the  strong Viterbo conjecture (\cite{GuRa20}),
states that $w_{\rm G}$ and $c^{\rm Z}$ coincide on convex  domains in $\mathbb{R}^{2n}$, that is,

\begin{guess}\label{conj:Sviterbo}
{\rm All normalized symplectic capacities coincide on convex  domains in $\mathbb{R}^{2n}$.}
\end{guess}


\begin{guess}[Viterbo \cite{Vi00}]\label{conj:viterbo}
{\rm On $\mathbb{R}^{2n}$, for any normalized symplectic capacity $c$ and any bounded convex domain $D$ there holds
\begin{equation}\label{e:viterbo}
\frac{c(D)}{c(B^{2n}(1))}\le \left(\frac{{\rm Vol}(D)}{{\rm Vol}(B^{2n}(1))}\right)^{1/n}
\end{equation}
(or equivalently $(c(D))^n\le{\rm Vol}(D,\omega_0^n)=n!{\rm Vol}(D)$),
with equality if and only if $D$ is symplectomorphic to the Euclidean ball, where ${\rm Vol}(D)$ denotes the Euclidean volume of $D$.}
\end{guess}

Since (\ref{e:viterbo}) is clearly true for $c=w_G$,
Conjecture~\ref{conj:viterbo} follows from Conjecture~\ref{conj:Sviterbo}.
Some special cases of Conjecture~\ref{conj:viterbo} were proved in \cite{Ba20, KaRo19}.

Surprisingly, Artstein-Avidan,  Karasev, and Ostrover \cite{AAKO14} showed that
Conjecture~\ref{conj:viterbo}  implies  the following long-standing famous conjecture
about the \textsf{ Mahler volume}
$$
M(\Delta):={\rm Vol}(\Delta\times\Delta^\circ)={\rm Vol}(\Delta){\rm Vol}(\Delta^\circ)
$$
of a bounded convex domain $\Delta\subset\mathbb{R}^n$ in convex geometry,
where $\Delta^{\circ}=\{x\in\mathbb{R}^{n}\,|\, \langle y, x\rangle\le 1\;\forall y\in \Delta\}$
is the polar  of $\Delta$.

\begin{guess}[{\bf Symmetric Mahler conjecture} \cite{Mah39}]\label{conj:RMahler}
{\rm $M(\Delta)\ge\frac{4^n}{n!}$ for any centrally symmetric bounded convex domain $\Delta\subset\mathbb{R}^n$.
}
\end{guess}

The $n=2$ case of this conjecture was proved by Mahler \cite{Mah39}. Iriyeh and Shibata \cite{IrSh20}
have very recently proved the $n=3$ case.
Some special classes of centrally symmetric bounded convex domains in $\mathbb{R}^n$,
for example, those with $1$-unconditional basis, zonoids, polytopes with at
most $2n+2$ facets, were proved to satisfy Conjecture~\ref{conj:RMahler} in
\cite{Schy08}, \cite{Re85} and \cite{LoRe98}, respectively. Karasev \cite{Ka19} recently confirmed 
the conjecture for hyperplane sections or projections of $l_p$-balls or the Hanner polytopes.
See \cite{ScWe83, Tao07} and the references of \cite{IrSh20} for more information.  

Hermann \cite{Her98} proved Conjecture~\ref{conj:Sviterbo} for convex Reinhardt domains $D$.
Recall that a subset $X$ of $\mathbb{C}^n$ is called a  \textsf{Reinhardt domain} (\cite{Her98}) if it is
invariant under the standard toric action $\mathbb{T}^n=\mathbb{R}^n/\mathbb{Z}^n$ on $\mathbb{C}^n$ defined by
\begin{equation}\label{e:toric-action}
(\theta_1,\cdots,\theta_n)\cdot(z_1,\cdots,z_n)=\left(e^{2\pi i\theta_1}z_1,\cdots,e^{2\pi i\theta_1}z_n\right).
\end{equation}
This is a Hamiltonian action (with respect to the standard symplectic structure $\omega_0$ on
$\mathbb{C}^n=\mathbb{R}^{2n}$) with the moment map
$$
\mu:\mathbb{C}^n\to \mathbb{R}^n,\;
(z_1,\cdots,z_n)\mapsto (\pi|z_1|^2,\cdots,\pi|z_n|^2)
$$
after identifying the dual of the Lie algebra of $\mathbb{T}^n$ with $\mathbb{R}^n$.

Let $\mathbb{R}^n_{\ge 0}$ (resp. $\mathbb{Z}^n_{\ge 0}$) denote the set of $x\in\mathbb{R}^n$ (resp. $x\in\mathbb{Z}^n$)
such that $x_i\ge 0$ for all $i=1,\ldots,n$.
Given a nonempty relative open subset $\Omega$ in $\mathbb{R}^n_{\ge 0}$ we call Reinhardt domains
$$
X_{\Omega}=\mu^{-1}(\Omega)\quad\hbox{and}\quad X_{\overline{\Omega}}=\mu^{-1}(\overline{\Omega})
$$
 \textsf{toric domains} associated to $\Omega$ and $\overline{\Omega}$ (the closure of $\Omega$), respectively.
 (Both $X_{\Omega}$ and $X_{\overline{\Omega}}$ have volumes ${\rm Vol}(\Omega)$ by \cite[Lemma~2.6]{Her98}.)
 Moreover, following \cite{GuH18}, if $\Omega$ is bounded, and
$$
\widehat{\Omega}=\left\{(x_1,\cdots,x_n)\in\mathbb{R}^n\,|\, (|x_1|,\cdots,|x_n|)\in\Omega\right\}
\quad\hbox{(resp. $\mathbb{R}^n_{\ge 0}\setminus {\Omega}$)}
$$
is convex (resp. concave) in $\mathbb{R}^n$, we said $X_{\Omega}$ and $X_{\overline{\Omega}}$ to be \textsf{convex toric domains}
(resp.  \textsf{concave toric domains}). There exists an equivalent definition in
\cite{RaSe19}. An open and bounded subset $A\subset\mathbb{R}^n$ is called a \textsf{balanced region} if
$[-|x_1|, |x_1|]\times\cdots\times [-|x_n|, |x_n|]\subset A$ for each $(x_1,\cdots, x_n)\in A$.
Such a set $A$ is determined by the relative open subset $|A|:=A\cap\mathbb{R}^n_{\ge 0}$ in $\mathbb{R}^n_{\ge 0}$.
For a nonempty relative open subset $\Omega$ in $\mathbb{R}^n_{\ge 0}$
there exists a balanced region $A\subset\mathbb{R}^n$ such that $\Omega=|A|$ if and only if
$[0, |x_1|]\times\cdots\times [0, |x_n|]\subset \Omega$ for each $(x_1,\cdots, x_n)\in \Omega$ (\cite[Remark~10]{RaSe19}).
The balanced region $A\subset\mathbb{R}^n$ is said to be
\textsf{convex} (resp.  \textsf{concave}) if $A$ (resp. $\mathbb{R}^n_{\ge 0}\setminus A$) is convex in  $\mathbb{R}^n$.
Then $X_{|A|}$ is \textsf{convex} (resp.  \textsf{concave}) in the sense above if and only if
the balanced region $A\subset\mathbb{R}^n$ is {convex} (resp.  {concave}). Clearly, the balanced regions are centrally symmetric,
and any convex or concave  balanced region is star-shaped.  By \cite[Lemma~2.5]{Her98} each convex  or concave toric domains is star-shaped.

By \cite[Examples~1.5, 1.12]{GuH18}, a $4$-dimensional toric domain $X_{\Omega}$ is convex (resp. concave)
if and only if
\begin{equation}\label{e:CCtoric}
\Omega=\{(x_1,x_2)\mid 0\le x_1\le a,\; 0\le x_2\le f(x_1)\}
\end{equation}
where $f:[0, a]\to \R_{\ge0}$ is a nonincreasing concave function (resp. convex function with $f(a)=0$).
(\textsf{Note that the concept of the present $4$-dimensional convex toric domain is stronger than one in }\cite{CG19}.)


Let $X_{\Omega}$ be a convex or concave toric domain associated to $\Omega\subset\mathbb{R}^n_{\ge 0}$ as above,
and let $\Sigma_\Omega$ and $\Sigma_{\overline{\Omega}}$ be the closures of the sets
$\partial\Omega\cap\mathbb{R}^n_{>0}$ and $\partial\overline{\Omega}\cap\mathbb{R}^n_{>0}$, respectively.
(Clearly,  $\Sigma_\Omega=\Sigma_{\overline{\Omega}}$.)
For $v\in\mathbb{R}^n_{\ge 0}$ we define
\begin{eqnarray}\label{e:norm1}
&&\|v\|_\Omega^\ast=\sup\{\langle v,w\rangle\,|\, w\in\Omega\}=\max\{\langle v,w\rangle\,|\, w\in\overline{\Omega}\}=
\|v\|_{\overline{\Omega}}^\ast, \\
&&[v]_\Omega=\min\{\langle v,w\rangle\,|\, w\in\Sigma_\Omega\}=\min\{\langle v,w\rangle\,|\, w\in\Sigma_{\overline{\Omega}}\}=
[v]_{\overline{\Omega}}^\ast\label{e:norm2}
\end{eqnarray}
(\cite[(1.9) and (1.13)]{GuH18}). 
Then $[v]_\Omega\le\|v\|_\Omega^\ast$, and  $\|v\|_{r\Omega}^\ast=r\|v\|_\Omega^\ast$ and $[v]_{r\Omega}=r[v]_\Omega$
for all $r>0$.

Recently,  Gutt and Ramos \cite{GuRa20} proved that all normalized symplectic capacities coincide
on any $4$-dimensional convex or concave toric domain, and that
$c_{\rm EH}$, $c_1^{\rm CH}$, $c_{\rm V}$ and $w_{\rm G}$ coincide
on any convex or concave toric domain. Combing the latter assertion
 with a result in \cite{GuH18} we can easily obtain
 the first result of this note, which claims that Conjecture~\ref{conj:Sviterbo} and therefore Conjecture~\ref{conj:viterbo}
holds true on all convex toric domains in $\mathbb{R}^{2n}$. More precisely, we have:

\begin{theorem}\label{th:main}
Let $\Omega\subset\mathbb{R}^n_{\ge 0}$ be a bounded nonempty relative open subset
 such that $\widehat{\Omega}$ is convex in $\mathbb{R}^n$.
 Then for any normalized symplectic capacity $c$ on $\mathbb{R}^{2n}$
convex toric domains $X_{\Omega}$ and $X_{\overline{\Omega}}$ have capacities
\begin{eqnarray*}
c(X_{\Omega})=c(X_{\overline{\Omega}})&=&\min\left\{\|v\|_{\Omega}^*\,\Big|\,v=(v_1,\cdots,v_n)\in\mathbb{Z}^n_{\ge 0},\; \sum_{i=1}^n v_i=1\right\}\\
&=&\min\{\|e_i\|_{\Omega}^\ast\,|\,i=1,\cdots,n\},
\end{eqnarray*}
where
$\{e_i\}_{i=1}^n$ is the standard orthogonal basis of $\mathbb{R}^n$.
\end{theorem}

It is unclear whether convex toric domains must be convex Reinhardt domains in $\mathbb{R}^{2n}$.
But the following Corollary~\ref{cor:main2} shows that Conjecture~\ref{conj:Sviterbo}
holds true for a class of convex domains in  $\mathbb{R}^{2n}$ that are not necessarily
Reinhardt domains.

%

\begin{corollary}\label{cor:main1}
Let $X_{\Omega_1}\subset\mathbb{R}^{2n}$ and $X_{\Omega_2}\subset\mathbb{R}^{2m}$ be convex toric domains
associated with bounded relative open subsets ${\Omega_1}\subset\mathbb{R}^{2n}_{\ge 0}$ and ${\Omega_2}\subset\mathbb{R}^{2m}_{\ge 0}$,
respectively. Then $X_{\Omega_1}\times X_{\Omega_2}$ is equal to the convex toric domain $X_{\Omega_1\times\Omega_2}$,
and for any normalized symplectic capacity $c$ on $\mathbb{R}^{2n+2m}$  there holds
$$
c(X_{\Omega_1}\times X_{\Omega_2})=\min\{c(X_{\Omega_1}),c(X_{\Omega_2})\}.
$$
The same conclusion holds true after ${\Omega_1}$ and ${\Omega_2}$ are replaced by $\overline{\Omega_1}$ and $\overline{\Omega_2}$,
respectively.
\end{corollary}
This is a direct consequence of \cite[(3.8)]{CHLS07} and Theorem~\ref{th:main}.
In Section~\ref{sec:Main} we shall prove it with only Theorem~\ref{th:main}.

For each $p\in [1,\infty]$ let  $\|\cdot\|_p$ denote the $l_p$-norm in $\mathbb{R}^n$ defined by
$$
\|x\|_p:=\left(\sum^n_{i=1}|x_i|^p\right)^{1/p}\;\hbox{if}\; p<\infty,\qquad
\|x\|_\infty:=\max_i|x_i|.
$$
Then the open unit ball  $B^n_p=\{x=(x_1,\cdots,x_n)\in\mathbb{R}^n\,|\, \|x\|_p<1\}$
is a convex balanced region in $\mathbb{R}^n$.
It was proved in \cite[Theorem~7]{RaSe19} that for a balanced region $A\subset\mathbb{R}^n$
there exists a symplectomorphism between $X_{4|A|}$ and the Lagrangian product $B^n_\infty\times_L A$
defined by
\begin{eqnarray*}
B^n_\infty\times_L A=\left\{(x_1,\cdots,x_n,y_1,\cdots,y_n)\in\mathbb{R}^{2n}\,|\, (x_1,\cdots,x_n)\in B^n_\infty,\,
(y_1,\cdots,y_n)\in A\right\},
\end{eqnarray*}
where $4|A|=\{(4x_1,\cdots, 4x_n)\,|\, (x_1,\cdots, x_n)\in|A|\}$. By this and Theorem~\ref{th:main} (resp.
Corollary~\ref{cor:main1}) we may, respectively, obtain two claims of the following

\begin{corollary}\label{cor:main2}
For a convex balanced region $A\subset\mathbb{R}^n$ and any normalized symplectic capacity $c$ on $\mathbb{R}^{2n}$ there holds
\begin{eqnarray*}
c(B^n_\infty\times_L A)=4\min\{\|e_i\|_{|A|}^\ast\,|\,i=1,\cdots,n\}.
\end{eqnarray*}
In particular, $c(B^n_p\times_L B^n_\infty)=c(B^n_\infty\times_L B^n_p)=4$ for every $p\in [1, \infty]$
(since the symplectomorphism $\mathbb{R}^{2n}\ni (x,y)\mapsto (-y, x)\in\mathbb{R}^{2n}$
maps $B^n_\infty\times_L B^n_p$ onto $B^n_\infty\times_L B^n_p$).
Moreover, for convex balanced regions $A_i\subset\mathbb{R}^{n_i}$, $i=1,\cdots,k$, it holds that
\begin{eqnarray*}
c((B^{n_1}_\infty\times\cdots\times B^{n_k}_\infty)\times_L (A_1\times\cdots\times A_k))=
\min_ic(B^{n_i}_\infty\times_L A_i).
\end{eqnarray*}
Consequently, the convex domain $(B^{n_1}_\infty\times\cdots\times B^{n_k}_\infty)\times_L (A_1\times\cdots\times A_k)$
satisfies Conjecture~\ref{conj:Sviterbo} and so Conjecture~\ref{conj:viterbo}
  by the first claim.
\end{corollary}

Clearly, this result is a partial generalization of \cite[Theorem~5.2]{Ba20} since
$B^n_\infty$ is equal to $\Box_n$ therein.
 Note that convex subsets $B^n_\infty\times_L B^n_p$ ($1\le p<\infty$) are not Reinhardt domains in $\mathbb{R}^{2n}$.

%
%

 Since $B^n_1$ is a convex balanced region in $\mathbb{R}^n$
and is equal to $(B^n_\infty)^\circ$,
Corollary~\ref{cor:main2} implies the known equality case in Mahler's conjecture,
which can also be proved by a straightforward computation because ${\rm Vol}(B^n_1)=2^n/n!$ and ${\rm Vol}(B^n_\infty)=2^n$
by (\ref{e:ball}).
This and Corollary~\ref{cor:main2} suggest the following questions for each $p\in (1, \infty)$:
Is Conjecture~\ref{conj:viterbo} for the convex domain $B^n_p\times (B^n_p)^\circ\subset\mathbb{R}^{2n}$ true?
Does Conjecture~\ref{conj:RMahler} for the ball $B^n_p$ hold true?

%

They are affirmative as examples of the following Theorems~\ref{th:main2}, \ref{th:Ray}, respectively.

\begin{theorem}[Saint-Raymond \cite{SR80}]\label{th:Ray}
Suppose that a centrally symmetric convex domain $K\subset\mathbb{R}^{n}$ is $1$-unconditional.
Then ${\rm Vol}(K\times K^{\circ})\geqslant\frac{4^n}{n!}$ and
equality holds if $K$ is a Hanner polytope.
\end{theorem}


Recall that in \cite{SR80, Re87, Schy08}
 a centrally symmetric convex domain $K\subset\mathbb{R}^{n}$ is called \textsf{$1$-unconditional}
if there exists a basis $\{\eta_1,\cdots,\eta_n\}$ of $\mathbb{R}^{n}$ such that
$$
\left\|\sum_{i=1}^n a_i\eta_i\right\|_K=\left\|\sum_{i=1}^n\varepsilon_ia_i\eta_i\right\|_K
$$
for all scalars $a_i\in\mathbb{R}$ and signs $\varepsilon_i\in\{-1,1\}$, $1\leqslant i\leqslant n$,
where $\|\cdot\|_K$ is the norm on $\mathbb{R}^n$ determined by $K$, that is,
$\|x\|_K=\min\{t\geqslant 0\,|\,x\in tK\}, x\in\mathbb{R}^n$.

For $1\le p\le\infty$, the \textsf{$p$-product} of
two centrally symmetric convex domains $K\subset\mathbb{R}^{n}$ and $M\subset\mathbb{R}^{m}$
is defined by
$$
K\times _pM:=\bigcup_{0\le t\le 1}\left((1-t)^{\frac{1}{p}}K\times t^{\frac{1}{p}}M\right),
$$
which is also centrally symmetric and has the corresponding norm
$$
\|(x,y)\|_{K\times _pM}=\left(\|x\|^p_K+ \|y\|_M^p\right)^{\frac{1}{p}},\quad (x,y)\in\mathbb{R}^n\times\mathbb{R}^m.
$$
From this it is not hard to derive that the operator $\times_p$ is associative.
Moreover, if both $K$ and $M$ are $1$-unconditional, so is $K\times _pM$.
Note also that $K\times_\infty M=K\times M$ and $K\times_1 M={\rm conv}\{(K\times\{0\})\cup(\{0\}\times M)\}$.
The $1$-product is also called \textsf{free sum}.

A centrally symmetric convex domain $K\subset\mathbb{R}^{n}$
is called a \textsf{Hanner polytope} if it is obtained by successively applying Cartesian products
and free sums to centered line segments in arbitrary order.
Hence every Hanner polytope in $\mathbb{R}^{n}$ is an affine image of $I\times_{p_1}\cdots\times_{p_{n-1}}I$,
 where $I=[-1,1]$ and $p_i\in\{1,\infty\}$, $i=1,\cdots,n-1$. 

It is not hard to check that both Hanner polytopes and closures of balanced regions
are $1$-unconditional convex domains. But a Hanner polytope is not necessarily
 balanced.

\begin{theorem}\label{th:main2}
Suppose that  $A\subset\mathbb{R}^{n}$ is $1$-unconditional convex domain.
Then $A\times_L A^{\circ}$ satisfies Conjecture~\ref{conj:viterbo}, precisely,
\begin{equation}\label{e:zon0}
c(A\times_L A^{\circ})\leqslant 4\leqslant (n!{\rm Vol}(A\times_L A^{\circ}))^{\frac{1}{n}}
\end{equation}
 for any normalized symplectic capacity $c$ on $\mathbb{R}^{2n}$.
\end{theorem}

Recall that an ellipsoid in an $n$-dimensional normed space $E$ is defined as
a subset $Q\subset E$ which is the image of $B^n_2$ by a line isomorphism
(cf. \cite[page~27]{Pi89}). We call the image of $B^n_p$ by a linear isomorphism of $\mathbb{R}^n$
a \textsf{$l_p$-ellipsoid} with $p\in [1, \infty]$.

 \begin{corollary}\label{cor:main3}
For a $l_p$-ellipsoid $Q=\Upsilon(B^n_p)\subset\mathbb{R}^n$ there holds
\begin{equation}\label{e:zon1}
c(Q\times_L Q^{\circ})=4\le (n!{\rm Vol}(Q\times_L Q^{\circ}))^{\frac{1}{n}}
\end{equation}
 for any normalized symplectic capacity $c$ on $\mathbb{R}^{2n}$.
In particular, Conjecture~\ref{conj:viterbo} holds for the convex domain $D=Q\times Q^\circ$.
\end{corollary}

Since the Mahler volume is affine invariant, $4\le (n!{\rm Vol}(Q\times_L Q^{\circ}))^{\frac{1}{n}}$
if and only if $4\leqslant (n!{\rm Vol}(B^n_p\times_L (B^n_p)^{\circ}))^{\frac{1}{n}}$.
The latter follows from (\ref{e:zon0}). In Section~\ref{sec:main3} we shall give a direct calculus proof of
the inequality.\\

\noindent{\bf Organization of the paper}.
In Section~\ref{sec:Main} we prove Theorem~\ref{th:main} and Corollary~\ref{cor:main1}.
 Next, we give proofs of Theorem~\ref{th:main2} and Corollary~\ref{cor:main3} in Section~\ref{sec:Main2}.
A direct proof of the Mahler conjecture for $l_p$-balls is given in Section~\ref{sec:main3}.
Finally, Section~\ref{sec:remark} includes some concluding remarks.\\

\noindent{\bf Acknowledgments}.
We would like to thank Dr. Matthias Schymura for telling us that the Mahler conjecture for $l_p$-balls
is true  as an example of a result by J. Saint-Raymond \cite{SR80}, and for sending
us his beautiful Diplomarbeit, which is very helpful to us because
 researches on the Mahler conjecture before March 2008 were explained explicitly.

\section{Proofs of Theorem~\ref{th:main} and Corollary~\ref{cor:main1}}\label{sec:Main}
\setcounter{equation}{0}

\begin{proof}[Proof of Theorem~\ref{th:main}]
By  \cite[Theorem 1.6]{GuH18} and \cite[Theorem 3.1]{GuRa20},
it holds that
\begin{equation}\label{e:main0}
w_{\rm G}(X_{\overline{\Omega}})=\min\left\{\|v\|_{\overline{\Omega}}^*\;\Big|\;v=(v_1,\cdots,v_n)\in\mathbb{Z}^n_{\ge 0},\; \sum_{i=1}^n v_i=1\right\}.
\end{equation}
Let $c$ be an arbitrarily given normalized symplectic capacity on $\mathbb{R}^{2n}$.
Then $c(X_{\overline{\Omega}})\geqslant w_{\rm G}(X_{\overline{\Omega}})$
by the normalization axiom of the symplectic capacity.
Next let us  show that
\begin{equation}\label{e:main1}
c(X_{\overline{\Omega}})\leqslant
\min\left\{\|v\|_{\overline{\Omega}}^*\;\Big|\;v=(v_1,\cdots,v_n)\in\mathbb{Z}^n_{\ge 0},\; \sum_{i=1}^n v_i=1\right\}.
\end{equation}
Let $\{e_i\}^n_{i=1}$ be the standard basis in $\mathbb{R}^n$, where
$e_i=(0,\cdots,0,1,0,\cdots,0)$ with only the $i$-th component non-zero, and equal to $1$,
 $i=1,\cdots,n$.
Write $L_i=\|e_i\|_{\overline{\Omega}}^*$ and define
$$
\overline{\Omega}_i^\star=\{x\in\mathbb{R}^n_{\ge 0} \,|\,\langle e_i,x\rangle\leqslant L_i\},\quad i=1,\cdots,n.
$$
Then for each $i$, $\overline{\Omega}\subset\overline{\Omega}_i^\star$ by the definition of $\|e_i\|_{\overline{\Omega}}^\ast$,
and there exists an obvious symplectomorphism from
   $X_{\overline{\Omega}_i^\star}=\{(z_1,\cdots,z_n)\in\mathbb{C}^n=\mathbb{R}^{2n}\,|\, \pi|z_i|^2\leqslant L_i\}$
onto $Z^{2n}(\sqrt{L_{i}/\pi})$.
It follows from the monotonicity and conformality of symplectic capacities that
\begin{eqnarray*}
c(X_{\overline{\Omega}})\le c(X_{\overline{\Omega}_i^\star})=c(Z^{2n}(\sqrt{L_{i}/\pi}))
=\frac{L_{i}}{\pi} c(Z^{2n}(1))=L_{i},\quad i=1,\cdots,n
\end{eqnarray*}
and so $c(X_{\overline{\Omega}})\le\min_iL_i$. Note that each vector
$v=(v_1,\cdots,v_n)\in\mathbb{Z}^n_{\ge 0}$ with $\sum_{i=1}^n v_i=1$
must have form $e_j$ for some $j\in\{1,\cdots,n\}$. We get (\ref{e:main1}) and therefore
\begin{eqnarray}\label{e:main2}
c(X_{\overline{\Omega}})&=&\min\left\{\|v\|_{\Omega}^*\,\Big|\,v=(v_1,\cdots,v_n)\in\mathbb{Z}^n_{\ge 0},\; \sum_{i=1}^n v_i=1\right\}
\nonumber\\
&=&\min\{\|e_i\|_{\Omega}^\ast\,|\,i=1,\cdots,n\}
\end{eqnarray}
since $\|v\|_{\overline{\Omega}}^\ast=\|v\|_\Omega^\ast$.

Finally, we also need to prove $c(X_{{\Omega}})=c(X_{\overline{\Omega}})$.
Clearly, $c(X_{{\Omega}})\le c(X_{\overline{\Omega}})$ by the monotonicity of symplectic capacities.
Since $X_{{\Omega}}$ is open and has the closure $X_{\overline{\Omega}}$ it follows
from the definition of the Gromov  width $w_G$ in (\ref{e:GromovWidth}) that
$w_G(X_{{\Omega}})=w_G(X_{\overline{\Omega}})$. This, and (\ref{e:main0}) and (\ref{e:main2}) yield
$$
c(X_{\overline{\Omega}})=w_G(X_{\overline{\Omega}})=w_G(X_{{\Omega}})\le
c(X_{{\Omega}})
$$
and hence $c(X_{{\Omega}})=c(X_{\overline{\Omega}})$.
Now the proof is complete.
\end{proof}

\begin{remark}\label{rm:1}
{\rm   Let $X_{\Omega}$ be a concave toric domain  associated to a relative open subset $\Omega\subset\mathbb{R}^n_{\ge 0}$.
 By  \cite[Theorem~1.14 \& Corollary~1.16]{GuH18} and \cite[Theorem~3.1]{GuRa20}, we have
\begin{eqnarray}\label{e:concave}
w_{\rm G}(X_{\overline{\Omega}})&=&c^{\rm CH}_1(X_{\overline{\Omega}})\nonumber\\
&=&\max\left\{[v]_{\overline{\Omega}}\;\Big|\;v=(v_1,\cdots,v_n)\in\mathbb{Z}^n_{>0},\; \sum_{i=1}^n v_i=n\right\}\nonumber\\
&=&\inf\left\{\sum^n_{i=1}w_i\;|\;
w=(w_1,\cdots,w_n)\in\partial\Omega\cap\mathbb{R}^{2n}_{>0}\right\}\nonumber\\
&=&\max\{\pi r^2\;|\;B^{2n}(r)\subset X_\Omega\}=w_{\rm G}(X_\Omega).
\end{eqnarray}
For any normalized symplectic capacity $c$ on $\mathbb{R}^{2n}$,
repeating the proof of Theorem~\ref{th:main} we  get
\begin{eqnarray*}
c(X_{\overline{\Omega}})&\le&\min\left\{\|v\|_{\Omega}^*\,\Big|\,v=(v_1,\cdots,v_n)\in\mathbb{Z}^n_{\ge 0},\; \sum_{i=1}^n v_i=1\right\}
\nonumber\\
&=&\min\{\|e_i\|_{\Omega}^\ast\,|\,i=1,\cdots,n\}.
\end{eqnarray*}
Clearly, we have also $c(X_{\overline{\Omega}})\le c(X_{\overline{{\rm conv}(\Omega)}})$
 and
 $$
c(X_{\overline{{\rm conv}(\Omega)}})\le\min\{\|e_i\|_{\overline{{\rm conv}(\Omega)}}^\ast\,|\,i=1,\cdots,n\}= \min\{\|e_i\|_{\Omega}^\ast\,|\,i=1,\cdots,n\}.
 $$
This final equality easily follows from (\ref{e:norm1}).

If $A\subset\mathbb{R}^n$ is a concave balanced region, since
the Lagrangian product $B^n_\infty\times_L A$ is symplectomorphic to $X_{4|A|}$
(\cite[Theorem~7]{RaSe19}), from (\ref{e:concave}) we get
\begin{eqnarray*}
w_{\rm G}(B^n_\infty\times_L A)=c^{\rm CH}_1(B^n_\infty\times_L A)=4
\inf\left\{\sum^n_{i=1}w_i\;|\;
w=(w_1,\cdots,w_n)\in(\partial|A|)\cap\mathbb{R}^{2n}_{>0}\right\}.
\end{eqnarray*}

 }
\end{remark}


\begin{proof}[Proof of Corollary~\ref{cor:main1}]
Since we can write
$$
X_{\Omega_1}=\{(z_{n+1},\cdots, z_{m+n})\in \mathbb{C}^m\,|\, (\pi|z_{n+1}|^2,\cdots, \pi|z_{m+n}|^2)\in\Omega_1\},
$$
then
\begin{eqnarray*}
X_{\Omega_1}\times X_{\Omega_2}&=&\{(z_1,\cdots, z_{m+n})\in \mathbb{C}^{n+m}\,|\,
(z_1,\cdots,z_n)\in X_{\Omega_1},\;(z_{n+1},\cdots, z_{m+n})\in X_{\Omega_2}\}\\
&=&X_{\Omega_1\times \Omega_2}
\end{eqnarray*}
and thus
 $c(X_{\Omega_1}\times X_{\Omega_2})=c(X_{\Omega_1\times \Omega_2})$.
By Theorem~\ref{th:main}, we  get
 $$
 c(X_{\Omega_1\times\Omega_2})=\min\{\|e_i\|_{\Omega_1\times\Omega_2}^\ast\,|\,i=1,\cdots, n+m\},
 $$
 where $\{e_i\}_{i=1}^{n+m}$ is the standard orthogonal basis of $\mathbb{R}^{n+m}$.
  But for $i=1,\cdots,n$,
\begin{eqnarray*}
\|e_i\|_{\Omega_1\times\Omega_2}^\ast&=&\sup\{\langle e_i, x\rangle\,|\,x=(x_1,\cdots,x_{n+m})\in\Omega_1\times\Omega_2\}\\
&=&\sup\{x_i\,|\,x=(x_1,\cdots,x_{n+m})\in\Omega_1\times\Omega_2\}\\
&=&\sup\{x_i\,|\,x=(x_1,\cdots,x_n)\in\Omega_1\}\\
&=&\|e_i\|_{\Omega_1}^\ast.
\end{eqnarray*}
Hence we arrive at
$$
\min\{\|e_i\|_{\Omega_1\times\Omega_2}^\ast\,|\,i=1,\cdots,n\}=\min\{\|e_i\|_{\Omega_1}^\ast\,|\, i=1,\cdots,n\}=c(X_{\Omega_1}).
$$
Similarly, we have
$\min\{\|e_i\|_{\Omega_1\times\Omega_2}^\ast\,|\, i=n+1,\cdots,n+m\}=c(X_{\Omega_2})$.
Therefore
$$
c(X_{\Omega_1\times\Omega_2})=\min\{c(X_{\Omega_1}),c(X_{\Omega_2})\}.
$$
This and Theorem~\ref{th:main} also lead to the second conclusion.
\end{proof}

\section{Proofs of Theorem~\ref{th:main2} and Corollary~\ref{cor:main3}}\label{sec:Main2}


\begin{proof}[Proof of Theorem~\ref{th:main2}]

We begin with the following lemma.

\begin{lemma}\label{lem:bal}
For a convex balanced region $A\subset\mathbb{R}^n$ and any normalized symplectic capacity $c$ on $\mathbb{R}^{2n}$,
 there holds
 $$
 c(A\times_L A^{\circ})\leqslant 4.
 $$
 \end{lemma}

\begin{proof}
Let  $r=\max\{\|e_i\|_{|A|}^\ast\,|\,i=1,\cdots,n\}$.
 By (\ref{e:norm1}) we deduce that $|A|\subset [0, r]^n$.
 This and the definition of the balanced region imply that $ A\subset rB^n_{\infty}$.
 It follows from the monotonicity and conformality of symplectic capacities that
  \begin{equation}\label{e:cap1}
  c(A\times_L A^{\circ})\leqslant c((rB^n_{\infty})\times_L A^{\circ})=r^2c(B^n_{\infty}\times_L (\frac{1}{r}A^{\circ})).
  \end{equation}

  Next, we claim that $A^{\circ}$ is also a convex balanced region.
  It suffices to prove that $A^{\circ}$ is a balanced region.
    In fact,  for any $(y_1,\cdots,y_n)\in A^{\circ}$,
  since $A$ is symmetric with respect to all coordinate hyperplanes, we have
  \begin{equation}\label{e:cap2}
  \{y_1,-y_1\}\times\{y_2,-y_2\}\times\cdots\times\{y_n,-y_n\}\in A^{\circ}.
  \end{equation}
  Moreover, for any $y,y'\in A^{\circ}$, we derive
  $$
  \langle ty+(1-t)y', x\rangle=t\langle y,x\rangle+(1-t)\langle y',x\rangle\leqslant1,\quad\forall x\in A,\;\forall 0<t<1,
  $$
   that is, $A^{\circ}$ is convex set. From this and (\ref{e:cap2}) we derive
    $$
    [-|y_1|,|y_1|]\times[-|y_2|,|y_2|]\times\cdots\times[-|y_n|,|y_n|]\in A^{\circ},
    $$
     namely, $A^\circ$ is a balanced region. 

Now from Corollary~\ref{cor:main2} and (\ref{e:cap1}) we deduce
\begin{eqnarray}\label{e:cap3}
 c(A\times_L A^{\circ})&\le& r^2c(B^n_{\infty}\times_L (\frac{1}{r}A^{\circ}))\nonumber\\
&=& 4r \min\{\|e_i\|_{|A^{\circ}|}^\ast|i=1,\cdots,n\}.
\end{eqnarray}
It remains to show that $\min\{\|e_i\|_{|A^{\circ}|}^\ast|i=1,\cdots,n\}\leqslant \frac{1}{r}$.
Let $r=\|e_j\|_{|A|}^\ast$ for some $1\le j\le n$.
Take  $a>0$  such that $ae_j\in |A|$. Then $\langle ae_j, x\rangle\le 1\;\forall x\in A^\circ$.
In particular, $\langle e_j, x\rangle\le\frac{1}{a}\;\forall x\in |A^\circ|$.
This shows $\|e_j\|_{|A^{\circ}|}^\ast\leqslant \frac{1}{a}$. Note that
$\|e_j\|_{|A|}^\ast>0$ and that $a>0$ can be chosen to be arbitrarily close to
$\|e_j\|_{|A|}^\ast$. We get
 $\|e_j\|_{|A^{\circ}|}^\ast\leqslant \frac{1}{\|e_j\|_{|A|}^\ast}=\frac{1}{r}$, and therefore
  $$
  \min\{\|e_i\|_{|A^{\circ}|}^\ast|i=1,\cdots,n\}\le\|e_j\|_{|A^{\circ}|}^\ast\le\frac{1}{r}.
  $$
This and (\ref{e:cap3}) lead to the desired result.
\end{proof}

By Theorem~\ref{th:Ray}, if a centrally symmetric convex domain $A\subset\mathbb{R}^{n}$ is
a balanced region, in particular a Hanner polytope, then  ${\rm Vol}(A\times_L A^{\circ})\geqslant\frac{4^n}{n!}$ and
therefore $A\times_L A^{\circ}$ satisfies Conjecture~\ref{conj:viterbo}, i.e.,
\begin{eqnarray}\label{e:cap4}
c(A\times_L A^{\circ})\leqslant 4\leqslant (n!{\rm Vol}(A\times_L A^{\circ}))^{\frac{1}{n}}
\end{eqnarray}
 for any normalized symplectic capacity $c$ on $\mathbb{R}^{2n}$.

Now assume  that  $A\subset\mathbb{R}^{n}$ is $1$-unconditional convex domain with basis $\{\eta_1,\cdots,\eta_n\}$.
Let $\{e_1,\cdots,e_n\}$ be the standard basis of $\mathbb{R}^{n}$, and let $\Upsilon\in{\rm GL}(n,\mathbb{R})$
map $\eta_i$ to $e_i$ for $i=1,\cdots,n$.
Since $\|x\|_{\Upsilon(A)}=\|\Upsilon^{-1}x\|_A$ for any $x\in\mathbb{R}^n$, a straightforward computation
shows that  $\Upsilon(A)\subset\mathbb{R}^{n}$ is $1$-unconditional convex domain with basis $\{e_1,\cdots, e_n\}$.
It follows that
$$
\|(x_1,\cdots,x_n)\|_{\Upsilon(A)}=\|(|x_1|,\cdots, |x_n|)\|_{\Upsilon(A)},\quad \forall x\in\mathbb{R}^n,
$$
which means that the convex domain $\Upsilon(A)\subset\mathbb{R}^{n}$ is
a balanced region. By (\ref{e:cap4}) we get
\begin{eqnarray}\label{e:cap5}
c(\Upsilon(A)\times_L (\Upsilon(A))^{\circ})\leqslant 4\leqslant (n!{\rm Vol}(\Upsilon(A)\times_L (\Upsilon(A))^{\circ}))^{\frac{1}{n}}
\end{eqnarray}
 for any normalized symplectic capacity $c$ on $\mathbb{R}^{2n}$.
Denote by $\Upsilon^T$  the transpose of $\Upsilon\in{\rm GL}(n,\mathbb{R})$ with respect to
the  inner product $\langle\cdot, \cdot\rangle$ in $\mathbb{R}^{n}$. Then
\begin{eqnarray}\label{e:cap6}
\Phi_\Upsilon:(\mathbb{R}^{2n},\omega_0)\to (\mathbb{R}^{2n},\omega_0),\;(x,y)\mapsto (\Upsilon x, (\Upsilon^T)^{-1}y)
\end{eqnarray}
is a symplectomorphism.
 By the definition of the polar it is easy to check that
\begin{eqnarray*}
(\Upsilon(A))^{\circ}=\{x\in\mathbb{R}^{n}\,|\, \langle y, x\rangle\le 1\;\forall y\in \Upsilon(A)\}
=\{(\Upsilon^T)^{-1}u\,|\, u\in A^\circ\}=(\Upsilon^T)^{-1}(A)^\circ.
\end{eqnarray*}
Then  $\Upsilon(A)\times (\Upsilon(A))^{\circ}=\Phi_\Upsilon(A\times A^\circ)$,
${\rm Vol}((\Upsilon(A))^{\circ})=|\det (\Upsilon^T)^{-1}|{\rm Vol}(A^\circ)$ and so
$$
{\rm Vol}(\Upsilon(A)\times (\Upsilon(A))^{\circ})={\rm Vol}(\Upsilon(A)){\rm Vol}((\Upsilon(A))^{\circ})
={\rm Vol}(A){\rm Vol}(A^\circ)={\rm Vol}(A\times_L A^\circ).
$$
From these and (\ref{e:cap5}) we derive (\ref{e:zon0}).
Theorem~\ref{th:main2} is proved.
\end{proof}

\begin{proof}[Proof of Corollary~\ref{cor:main3}]
Since every closed $l_p$-ball $\overline{B^n_p}$
is a $1$-unconditional convex domain with basis $\{e_i\}^n_{i=1}$ in $\mathbb{R}^n$,
 for any normalized symplectic capacity $c$ on $\mathbb{R}^{2n}$ we derive from  (\ref{e:zon0})  that
\begin{equation}\label{e:zon2}
c(B^n_p\times_L (B^n_p)^{\circ})\leqslant 4\leqslant (n!{\rm Vol}(B^n_p\times_L (B^n_p)^{\circ}))^{\frac{1}{n}}
\end{equation}
and therefore 
\begin{equation}\label{e:zon3}
c(A\times_L A^{\circ})\leqslant 4\leqslant (n!{\rm Vol}(A\times_L A^{\circ}))^{\frac{1}{n}}.
\end{equation}
If $p=1$ or $\infty$, Corollary~\ref{cor:main2} has yielded $c(B^n_p\times_L (B^n_p)^{\circ})=4$.
For $1<p<\infty$, we have $w_G(B^n_p\times_L (B^n_p)^{\circ})\ge 4$ by \cite[Proposition~3.1]{KaRo19}.
As above these give rise to 
$$
c(A\times_L A^{\circ})\ge w_G(A\times_L A^{\circ})=w_G(B^n_p\times_L (B^n_p)^{\circ})\ge 4
\quad\forall p\in [1, \infty].
$$ 
This and the first inequality in (\ref{e:zon3}) lead to
 equality in (\ref{e:zon1}).
\end{proof}

%

\section{A direct proof of the Mahler conjecture for $l_p$-balls}\label{sec:main3}


In this section we shall prove the following.

\begin{theorem}\label{th:3main}
Let $Q=\Upsilon(B^n_p)\subset\mathbb{R}^n$ be a $l_p$-ellipsoid with $\Upsilon\in{\rm GL}(n,\mathbb{R})$. If $n=1$ then
${\rm Vol}(Q\times Q^\circ)={\rm Vol}(Q){\rm Vol}(Q^\circ)\equiv 4$
for all $p\in[1,\infty]$. If $n\ge 2$ then there holds
\begin{eqnarray}\label{e:viterbo2}
{\rm Vol}(Q\times Q^\circ)={\rm Vol}(Q){\rm Vol}(Q^\circ)\ge \frac{4^n}{n!}
\end{eqnarray}
for all $p\in[1,\infty]$, and the equality  holds if and only if $p=1$ or $p=\infty$.
\end{theorem}

As the arguments below (\ref{e:cap6})
we only need to prove the case $\Upsilon=id_{\mathbb{R}^n}$, that is:

\begin{claim}\label{cl:Ma1}
For $n=1$, ${\rm Vol}(B^n_p\times (B^n_p)^\circ)={\rm Vol}(B^n_1\times (B^n_1)^\circ)=4\;\forall p\in[1,\infty]$.
If $n\ge 2$ then
\begin{eqnarray}\label{e:Ma0}
{\rm Vol}(B^n_p\times (B^n_p)^\circ)={\rm Vol}(B^n_p){\rm Vol}((B^n_p)^\circ)\ge 4^n/n!,\quad\forall p\in[1,\infty],
\end{eqnarray}
and the equality in (\ref{e:Ma0}) holds if and only if $p=1$ or $p=\infty$.
\end{claim}

This is a special example of Theorem~\ref{th:Ray} because $B^n_p$ is
a centrally symmetric convex domain $\mathbb{R}^{n}$ with $1$-unconditional
basis $\{e_i\}^n_{i=1}$. However, we here give a simple calculus  proof of it.

Since $(B^n_p)^\circ=B^n_q$ with $q=p/(p-1)$, and $[1,2]\ni p\mapsto q=p/(p-1)\in [2,\infty]$
is a homeomorphism, by symmetry it suffices to prove Claim~\ref{cl:Ma1} for $p\in [1, 2]$.

By \cite[(1.17)]{Pi89},  we have
\begin{eqnarray}\label{e:ball}
{\rm Vol}(B^n_p)=\left(2\Gamma\left(1+\frac{1}{p}\right)\right)^n\left(\Gamma\left(1+\frac{n}{p}\right)\right)^{-1}
\end{eqnarray}
and so
\begin{eqnarray*}
{\rm Vol}((B^n_p)^\circ)&=&\left(2\Gamma\left(1+\frac{1}{p/(p-1)}\right)\right)^n\left(\Gamma\left(1+\frac{n}{p/(p-1)}\right)\right)^{-1}\\
&=&\left(2\Gamma\left(2-\frac{1}{p}\right)\right)^n\left(\Gamma\left(n+1-\frac{n}{p}\right)\right)^{-1}
\end{eqnarray*}
and
\begin{eqnarray*}
{\rm Vol}(B^n_p){\rm Vol}((B^n_p)^\circ)
&=&\frac{4^n\left(\Gamma\left(1+\frac{1}{p}\right)\right)^n
\left(\Gamma\left(2-\frac{1}{p}\right)\right)^n}
{\Gamma\left(1+\frac{n}{p}\right)\Gamma\left(n+1-\frac{n}{p}\right)}.
\end{eqnarray*}
Taking the derivative of the function  $[1,2]\ni p\mapsto {\rm Vol}(B^n_p){\rm Vol}((B^n_p)^\circ)$ we get
\begin{eqnarray*}
&&\frac{d}{dp}{\rm Vol}(B^n_p){\rm Vol}((B^n_p)^\circ)\\
&=&4^n\frac{n}{p^2}\Gamma(1+\frac{1}{p})^{n-1}\Gamma(2-\frac{1}{p})^{n-1}
\frac{[\Gamma(1+\frac{1}{p})\Gamma'(2-\frac{1}{p})-\Gamma'(1+\frac{1}{p})\Gamma(2-\frac{1}{p})]}
{\Gamma(1+\frac{n}{p})\Gamma(n+1-\frac{n}{p})}  \\
&&+4^n\frac{n}{p^2}\Gamma(1+\frac{1}{p})^n\Gamma(2-\frac{1}{p})^n\frac{[\Gamma'(1+\frac{n}{p})\Gamma(n+1-\frac{n}{p})-
\Gamma'(n+1-\frac{n}{p})\Gamma(1+\frac{n}{p})]}{\Gamma(1+\frac{n}{p})^2\Gamma(n+1-\frac{n}{p})^2}.
\end{eqnarray*}
Recall that  the formula $\Gamma'(x)=\Gamma(x)\psi(x)\;\forall x>0$, where $\psi$-function is defined by
\begin{eqnarray*}
\psi(x)&=&\lim_{n\to\infty}\left\{\ln n-\sum^n_{k=0}\frac{1}{x+k}\right\}\\
&=&\int^\infty_0[e^{-t}-(1+t)^{-x}]t^{-1}dt\quad\hbox{(Gauss intergral formula)}\\
&=&-\gamma+ \int_0^1\frac{1-t^{x-1}}{1-t}dt\quad\hbox{(Dirichlet formula)}\\
\end{eqnarray*}
 where $\gamma$ is Euler constant.
We can  immediately deduce
\begin{eqnarray*}
&&\frac{d}{dp}{\rm Vol}(B^n_p){\rm Vol}((B^n_p)^\circ)\\
&=&4^n\frac{n}{p^2}\Gamma(1+\frac{1}{p})^n\Gamma(2-\frac{1}{p})^n\frac{\psi(2-\frac{1}{p})-\psi(1+\frac{1}{p})+\psi(1+\frac{n}{p})-\psi(n+1-\frac{n}{p})}{\Gamma(1+\frac{n}{p})\Gamma(n+1-\frac{n}{p})}\\
&=&\frac{n}{p^2}\left[\psi(2-\frac{1}{p})-\psi(1+\frac{1}{p})+\psi(1+\frac{n}{p})-\psi(n+1-\frac{n}{p})\right]{\rm Vol}(B^n_p){\rm Vol}((B^n_p)^\circ).
\end{eqnarray*}
Denote by $\Phi_n(p)$ the function in the square brackets. Then $\Phi_1(p)\equiv 0$ and so
the first conclusion in Claim~\ref{cl:Ma1} holds true.

\begin{claim}\label{cl:Ma2}
When $n\ge 2$, $\Phi_n(2)=0$ and $\Phi_n(p)>0$ for any $1\le p<2$.
\end{claim}

We first admit this. Then the function  $[1,2]\ni p\mapsto {\rm Vol}(B^n_p){\rm Vol}((B^n_p)^\circ)$
is strictly monotonously increasing for each integer $n\ge 2$.
Moreover, ${\rm Vol}(B^n_1)=2^n/n!$ and ${\rm Vol}(B^n_\infty)=2^n$.
Claim~\ref{cl:Ma2} immediately leads to the second conclusion in Claim~\ref{cl:Ma1}.

\begin{proof}[Proof of Claim~\ref{cl:Ma2}]
Since  $\psi(x+1)=\psi(x)+\frac{1}{x}$, then $\Phi_n(2)=0$ and
$\Phi_n(1)=\sum_{k=2}^{n}\frac{1}{k}$.
We always assume $1<p<2$ below.
By the Dirichlet formula above we get
$$
\psi(s+1)=-\gamma+\int_0^1\frac{1-x^s}{1-x}dx.
$$
It follows that
\begin{eqnarray}\label{e:Ma1}
\psi(2-\frac{1}{p})-\psi(1+\frac{1}{p})&=&\int_0^1\frac{x^{1/p}-x^{1-1/p}}{1-x}dx,\\
\psi(1+\frac{n}{p})-\psi(n+1-\frac{n}{p})&=&\int_0^1\frac{x^{n-n/p}-x^{n/p}}{1-x}dx\nonumber\\
&=&\int_0^1\frac{y^{1-1/p}-y^{1/p}}{1-y}\frac{1-y}{1-y^{1/n}}\frac{1}{n}y^{\frac{1}{n}-1}dy\label{e:Ma2}
\end{eqnarray}
by setting $x^n=y$. For convenience let $a=1/n$ and
$$
f(y):=\frac{1-y}{1-y^{1/n}}\frac{1}{n}y^{\frac{1}{n}-1}=a\frac{1-y}{1-y^{a}}y^{a-1}.
$$
 A straightforward computation leads to
\begin{eqnarray*}
f'(y)&=&a\left(\frac{1-y}{1-y^{a}}\right)'y^{a-1}+ a\frac{1-y}{1-y^{a}}(a-1)y^{a-2}\\
&=&ay^{a-1}\frac{-(1-y^a)-(1-y)(-ay^{a-1})}{(1-y^{a})^2}+a\frac{1-y}{1-y^{a}}(a-1)y^{a-2}\\
&=&\frac{ay^{a-1}}{(1-y^{a})^2}\left(-(1-y^{a})+ay^{a-1}(1-y)+(1-y)(a-1)y^{-1}(1-y^{a})\right)\\
&=&\frac{ay^{a-1}}{(1-y^{a})^2}\left((a-1)(y^{-1}-1)+y^{a-1}-1\right)\\
&=&\frac{ay^{a-1}}{(1-y^{a})^2}\left(\frac{1}{y}(a- ay-1+y^{a})\right).
\end{eqnarray*}
Let $g(y)=a- ay-1+y^{a}$. Then $g(0)=a-1<0$, $g(1)=0$ and $g'(y)=-a+ ay^{a-1}>0$ for all $0<y<1$.
It follows that $g(y)<0$ and so $f'(y)<0$ for all $0<y<1$.

On the other hand, by L'Hospital rule, we get $\lim_{y\to 1}f(y)=1$.
Hence $f(y)>1$ for $0<y<1$.  Using (\ref{e:Ma1}) and (\ref{e:Ma2}) we deduce that
 \begin{eqnarray*}
\Phi_n(p)&=&\psi(2-\frac{1}{p})-\psi(1+\frac{1}{p})+ \psi(1+\frac{n}{p})-\psi(n+1-\frac{n}{p})\\
&=&\int_0^1\frac{y^{1/p}-y^{1-1/p}}{1-y}dy+
\int_0^1\frac{y^{1-1/p}-y^{1/p}}{1-y}f(y)dy\\
&=&\int_0^1\frac{y^{1-1/p}-y^{1/p}}{1-y}(f(y)-1)dy>0
\end{eqnarray*}
since $y^{1-1/p}-y^{1/p}=y^{1/p}(y^{1-2/p}-1)<0$ for $0<y<1$ and $1<p<2$.
Claim~\ref{cl:Ma2} is proved.
\end{proof}

\section{Concluding remarks}\label{sec:remark}
\setcounter{equation}{0}

\begin{remark}\label{rm:2}
{\rm For $1\leqslant p<\infty$, $\mathbb{X}_p=\{(x,y)\in\mathbb{R}^2\times\mathbb{R}^2\,|\,\|x\|^p+\|y\|^p \leqslant1\}$
is called the $l_p$-sum of two Langrangian  open unit discs $B^2_2$, where  $\|\cdot\|$ denotes the standard Euclidean norm on $\mathbb{R}^2$. If $p=\infty$, $\mathbb{X}_{\infty}=\{(x,y)\in\mathbb{R}^2\times\mathbb{R}^2\,|\,\max\{\|x\|, \|y\|\}<1\}$
 is exactly the Lagrangian product $B^2_2\times_L B^2_2$.
 For any normalized symplectic capacity $c$ on $\mathbb{R}^4$ and $p\in[1,\infty]$
 it easily follows from  \cite{OsRa19, Ra15, GuH18, GuRa20} that
\begin{equation}\label{e:4.1}
c(\mathbb{X}_p)=\left\{
\begin{aligned}
&2\pi(1/4)^{1/p}, & p\in[1,2],\\
&\frac{4\Gamma(1+ \frac{1}{p})^2}{\Gamma(1+\frac{2}{p})}, & p\in[2,\infty),\\
&4,  & p=\infty.\\
\end{aligned}
\right.
\end{equation}
In particular, $\mathbb{X}_p$ satisfies Conjecture~\ref{conj:Sviterbo}.

In fact, for $p\in[1,\infty)$, by \cite[Theorem~5]{OsRa19}
  $\mathbb{X}_p$ is symplectomorphic  to $X_{\Omega_p}$, where $\Omega_p$ is the relatively
open set in $\mathbb{R}_{\geqslant0}^2$ bounded by the coordinate axes and the curve $\gamma_p$ parametrized by
$$(2\pi v+g_p(v), g_p(v)), \text{ for } v\in[0,(1/4)^{1/p}], $$
$$(g_p(-v), -2\pi v+ g_p(-v)), \text{ for } v\in[-(1/4)^{1/p},0],$$
where $g_p:[0,(1/4)^{1/p}]\rightarrow\mathbb{R}$ is the function defined by
 $$
 g_p(v):=2\int_{(\frac{1}{2}+\sqrt{\frac{1}{4}-v^p})^{1/p}}^{\frac{1}{2}-\sqrt{\frac{1}{4}-v^p}^{1/p}}\sqrt{(1-r^{p})^{2/p}-\frac{v^2}{r^2}}dr.
 $$
For $p=\infty$, Theorem~3 in \cite{Ra15} (with the notations in \cite[Theorem~6]{OsRa19})
claimed that
 $\mathbb{X}_{\infty}=B^2_2\times_L B^2_2$ is symplectictomorphic to $X_{\Omega_\infty}$, where $\Omega_{\infty}$ is the
the relatively open set in $\mathbb{R}_{\geqslant0}^2$
 bounded by the coordinate axes and the curve $\gamma_{\infty}$ parametrized by
 $$
 2(\sqrt{1-v^2}+v(\pi- {\rm arccos}v), \sqrt{1-v^2}-v{\rm arccos}v), \text{ for } v\in[-1,1]
 $$
(or equivaliently, $(2\sin(\alpha/2)-\alpha\cos(\alpha/2), 2\sin(\alpha/2)+ (2\pi-\alpha)\cos(\alpha/2))$ with $\alpha\in [0,2\pi]$,
 see \cite[Theorem~3]{Ra15}).
Moreover, by \cite[Proposition~8]{OsRa19}, we also know that the toric domain  $X_{\Omega_p}$
is convex for $p\in [1,2]$, and concave for $p\in [2,\infty]$.
Hence for any normalized symplectic capacity $c$ on $\mathbb{R}^4$, \cite[Theorem~1.4]{GuRa20} and \cite[Theorem~1]{OsRa19} lead to
the first two cases in (\ref{e:4.1}),
and the third case follows from  \cite[Theorem~1.4]{GuRa20} and \cite[Theorem 1.14]{GuH18},
 \begin{eqnarray*}
c(\mathbb{X}_{\infty})&=&\max\{[v]_{\Omega_\infty}\,|\, v\in\mathbb{Z}^2_{>0},\;\sum_iv_i=2\}\\
&=&\inf\left\{w_1+w_2\;|\;
w=(w_1, w_2)\in\partial\Omega_\infty\cap\mathbb{R}^{2}_{>0}\right\}=4.
\end{eqnarray*}
 }
\end{remark}

\begin{remark}\label{rm:3}
{\rm  Suppose that each of symplectic manifolds $X^{(1)},\cdots, X^{(m)}$ is either a convex
 toric domain or $4$-dimensional concave toric domain or equal to $\mathbb{X}_p$
 as in (\ref{e:4.1}).  Since each convex  or concave toric domain or  $\mathbb{X}_p$ is star-shaped,
 then for any normalized symplectic capacity $c$ on $\mathbb{R}^{2n}$ with $2n=\sum^m_{i=1}\dim X^{(i)}$,
  from \cite[(3.8)]{CHLS07}, Theorem~\ref{th:main} and (\ref{e:4.1})
 we derive
  $$
c_1^{\rm EH}(\prod^m_{i=1}X^{(i)})=\min_i c(X^{(i)}).
$$
 }
\end{remark}

\begin{remark}\label{rm:4}
{\rm The main result of \cite{AAKO14} is $c_{\rm EHZ}(\Delta\times\Delta^\circ)=4$
for any  bounded convex domain $\Delta\subset\mathbb{R}^n$. By this and  
(\ref{e:zon0}) and (\ref{e:zon1}) it seems to be reasonable to conjecture that
$c(\Delta\times\Delta^\circ)=4$ for any normalized symplectic capacity $c$ on $\mathbb{R}^{2n}$
and any bounded convex domain $\Delta\subset\mathbb{R}^n$.
 }
\end{remark}

%
%

\medskip
\begin{tabular}{l}
Laboratory of Mathematics and Complex Systems (Ministry of Education), \\
School of Mathematical Sciences, Beijing Normal University,\\
 Beijing 100875, People's Republic of China\\
 E-mail address: shikun@mail.bnu.edu.cn,\hspace{5mm}gclu@bnu.edu.cn\\
\end{tabular}


\end{document}